\documentclass[12pt]{article} 
\usepackage{amssymb, amsmath, amsthm}
\usepackage{geometry} 
\usepackage{amsfonts} 
\usepackage{color}
\usepackage{cancel}
\usepackage{mathrsfs}
\usepackage{yfonts}
\usepackage{accents}
\usepackage{commath}
\usepackage{relsize}
\usepackage[pdftex]{graphicx}
\usepackage[utf8]{inputenc}
\usepackage{cite}
\usepackage{multicol}
\usepackage{tabularx}
\usepackage{multirow}
\usepackage{float}
 \usepackage{relsize}
\usepackage{array}
\usepackage{adjustbox}
\usepackage{tikz}
\usepackage{comment}
\usepackage{color}

\newcommand{\lamkep}{\lambda_{k,\varepsilon}}
\newcommand{\lamk}{\lambda_{k}}
\newcommand{\sig}{\sigma}
\newcommand{\Gep}{\sqrt{G_{\e}} d\xi d\tau}
\newcommand{\G}{\sqrt{G_{0}} d\xi}
\newcommand{\Gt}{\sqrt{G_{0}} d\xi d\tau}
\newcommand{\intg}{\int_{\Gamma}}
\newcommand{\pa}{\partial}
\newcommand{\efunc}{\widetilde{\Phi}_{k,\e}}
\newcommand{\intp}{\int^{1}_{0}\int_{\Gamma}}
\newcommand{\intm}{\int^{0}_{-1}\int_{\Gamma}}
\newcommand{\inta}{\int^{1}_{-1}\int_{\Gamma}}
\providecommand{\norm}[1]{\l#1\|}

\newcommand{\e}{\varepsilon}

\newcommand{\Rn}{\mathbb{R}^n}

\newcommand{\ct}[1]{\langle {#1}\rangle \lower.3ex\hbox{$_{t}$}}
\newcommand{\lt}[1]{[ {#1}] \lower.3ex\hbox{$_{t}$}}

\newcommand{\N}{{\mathbb N}}

\usepackage{bbm}

\newtheorem{thm}{Theorem}[section]
\newtheorem{lem}[thm]{Lemma}
\newtheorem{rmk}[thm]{Remark}

\title{\LARGE{\bf Two-phase eigenvalue problem on thin domains with Neumann boundary condition}\thanks{This research was partially supported by the Grant-in-Aid for Scientific Research (B) (\#26287020) 
Japan Society for the Promotion of Science.}}

\author{Toshiaki Yachimura\thanks{
Research Center for Pure and Applied Mathematics, Graduate
School of Information Sciences, Tohoku University, Sendai 980-8579, Japan.
{\em Electronic mail address:}
yachimura@ims.is.tohoku.ac.jp}}
\date{}

\begin{document}
\maketitle

\begin{abstract}
In this paper, we study an eigenvalue problem with piecewise constant coefficients on thin domains with Neumann boundary condition, and we analyze the asymptotic behavior of each eigenvalue as the domain degenerates into a certain hypersurface being the set of discontinuities of the coefficients. We show how the discontinuity of the coefficients and the geometric shape of the interface affect the asymptotic behavior of the eigenvalues by using a variational approach. 
\end{abstract}

\bigskip

\bigskip


\section{Introduction and main results}
In this paper, we study an eigenvalue problem with piecewise constant coefficients on thin domains with Neumann boundary condition, and we analyze the asymptotic behavior of each eigenvalue as the domain degenerates into a certain hypersurface being the set of discontinuities of the coefficients. Physically speaking, the problem dealt with in this paper is to consider the frequency of the composite material when two different materials are joined thinly. 
This problm is also related to the heat diffusion on thin heat conductors. 

We will formulate the two phase eigenvalue problem on thin domains. Let $D \subset \mathbb{R}^n$ $(n \geqslant 2)$ be a bounded domain whose boundary $\Gamma$ is of class $C^{2}$ and connected. For sufficiently small $\e > 0$, we put 
$$\Omega_{-}\left(\e\right) = \{ x \in D \,\, | \,\, d\left(x,\Gamma\right) < \e \}, \qquad \Omega_{+}\left(\e\right) = \{ x \in \mathbb{R}^n \setminus \overline{D} \,\, | \,\, d\left(x,\Gamma\right) < \e \}, $$where $d\left(x,\Gamma\right)$ denotes the Euclidean distance from $x$ to $\Gamma$. We define 
$$\Omega\left(\e\right) = \Omega_{-}\left(\e\right) \cup \Omega_{+}\left(\e\right) \cup \Gamma. $$  
We denote by $\nu$ the outward unit normal vector to the boundary $\pa \Omega(\e)$ and by $\nu_{\Gamma}$ the outward unit normal vector to the interface $\Gamma$. 

Now we will consider a two-phase eigenvalue problem on $\Omega(\e)$ as follows:
\begin{equation}\label{P}
\begin{cases}
-\mathrm{div} \left(\sigma \nabla \Phi \right) = \lambda \Phi \hspace{-0.1cm} &\text{in} \,\, \Omega(\e), \\
\displaystyle \frac{\partial \Phi}{\partial \nu} = 0 \, &\text{on} \, \partial \Omega(\e),
\end{cases}
\end{equation}
where $\sigma = \sigma(x) \left( x \in \Omega(\e) \right)$ is a piecewise constant function given by 
\begin{equation*}
\sigma(x) = \begin{cases}
\sig_{-}, \quad &x \in \Omega_{-}(\e) \cup \Gamma,\\
\sig_{+}, \quad &x \in \Omega_{+}(\e),
\end{cases}
\end{equation*}
and $\sig_{-}$, $\sig_{+}$ are distinct positive constants (i.e. $\sig_{-}, \sig_{+} > 0, \sig_{-} \neq \sig_{+}$). 
\begin{figure}[H]
\centering
\includegraphics[width=6cm]{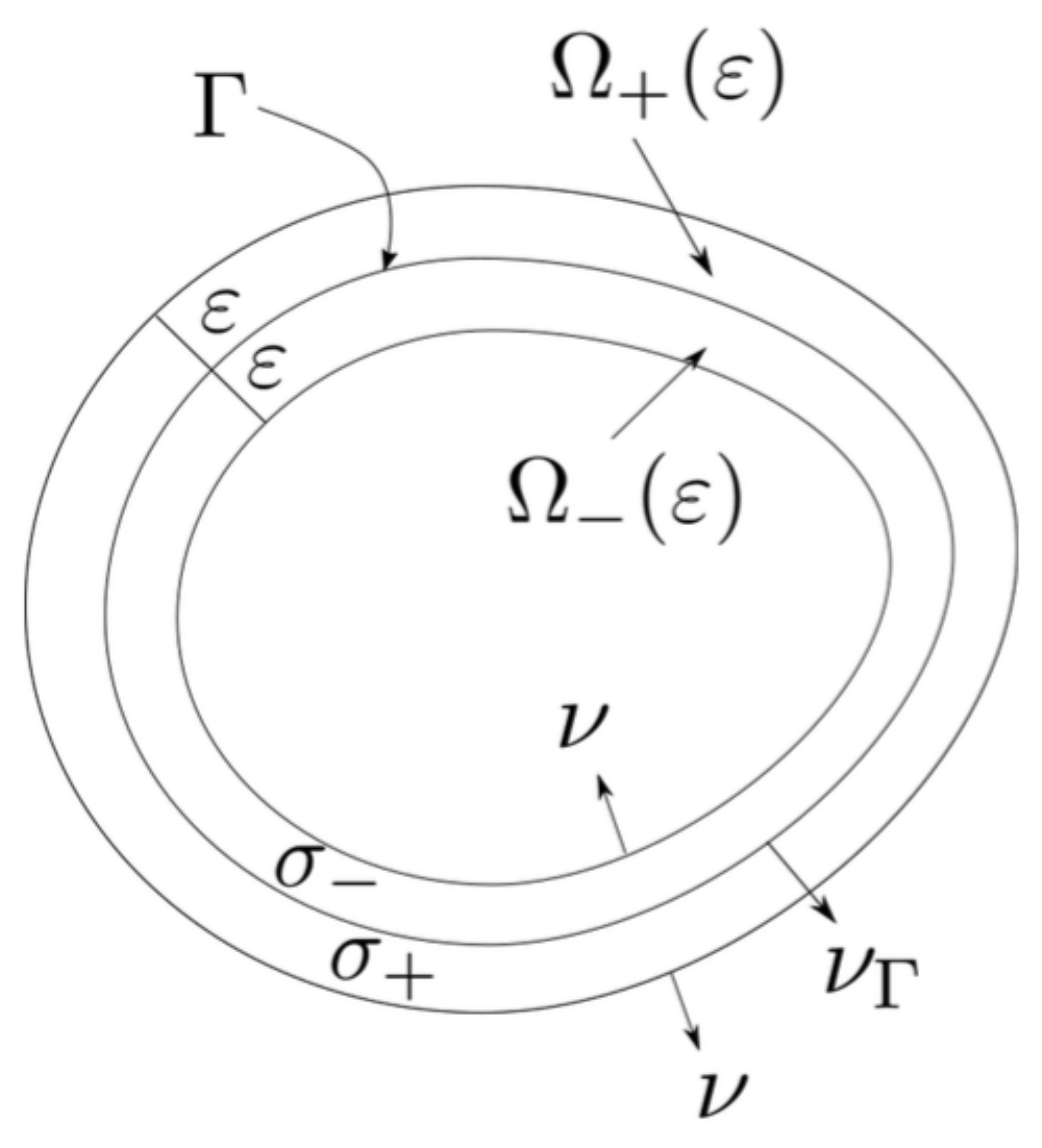}
\caption{Our problem setting}
\end{figure}
We consider the problem \eqref{P} in a weak sense, namely, $\lambda \in \mathbb{C}$ is an eigenvalue of \eqref{P} if there exists $\Phi \in H^{1}(\Omega(\e))$ such that $\Phi \not \equiv 0$ and 
\begin{equation}\label{weak}
\int_{\Omega(\e)} \sig \nabla \Phi \cdot \nabla \psi dx = \lambda \int_{\Omega(\e)} \Phi \psi dx \quad \text{for any } \psi \in H^{1}(\Omega(\e)).
\end{equation}
By a standard argument of self-adjoint operators, the eigenvalues of \eqref{P} are non-negative real numbers and the set of all eigenvalues is discrete. 
Let $\{ \lamkep \}_{k\geqslant1}$ be the eigenvalues satisfying $0 = \lambda_{1,\e} < \lambda_{2,\e} \leqslant \lambda_{3,\e} \leqslant \cdots \to +\infty$ and $\{ \Phi_{k,\e} \}_{k\geqslant1}$ be the associated eigenfunctions in \eqref{P}. 

Since $\sig = \sig(x)$ $(x \in \Omega(\e))$ is a piecewise constant function, we can rewrite \eqref{weak} as follows:
\begin{equation}\label{pb2}
\begin{cases}
- \sig_{\pm} \Delta \Phi_{\pm} = \lambda \Phi_{\pm} \quad &\text{in} \,\, \Omega_{\pm}(\e), \\
\Phi_{-} = \Phi_{+} \quad &\text{on} \,\, \Gamma, \\
\vspace{0.15cm}
\displaystyle \sig_{-} \frac{\partial \Phi_{-}}{\partial \nu_{\Gamma}} = \sig_{+} \frac{\partial \Phi_{+}}{\partial \nu_{\Gamma}} \quad &\text{on} \,\, \Gamma, \\
\displaystyle \frac{\partial \Phi_{\pm}}{\partial \nu} = 0 \quad &\text{on} \,\, \pa \Omega_{\pm}(\e) \setminus \Gamma.
\end{cases} 
\end{equation}
Here $\Phi_{-}$ and $\Phi_{+}$ are the restriction of $\Phi$ on $\Omega_{-}$ and $\Omega_{+}$, respectively. The third equality of \eqref{pb2} is usually called {\em transmission condition}, which can be interpreted as the continuity of the flux through the interface $\Gamma$.  
The purpose of this paper is to consider the asymptotic behavior of the eigenvalues $\{ \lamkep \}_{k\geqslant1}$ as $\e \to 0$ when the domain $\Omega(\e)$ degenerates to the interface $\Gamma$. In particular, our aims are to show how the discontinuity of the coefficients and the geometric shape of the interface $\Gamma$ affect the asymptotic behavior of eigenvalues as the domain $\Omega(\e)$ degenerates to the interface $\Gamma$. 

The study of domain perturbation and two-phase eigenvalue problems arises in some problems in the material science. For example, consider the problem of coating a material with a different one. Such a problem is called a {\em reinforcement problem}, and many authors have studied this kind of problem. We refer to \cite{Fr}\cite{GLNP}\cite{RW}.
Also in the case of the Laplacian (i.e. $\sig_{-} = \sig_{+}$), since pioneering work of Courant and Hilbert \cite{CH}, there have been many studies in various situations (e.g. dumbbell shaped domain\cite{A}\cite{J}\cite{JKo}, domain with small holes\cite{Oz}\cite{RT} and the references therein). 

Especially on thin domains, the geometric shape of the hypersurface where they degenerate affects the asymptotic behavior of eigenvalues of the Laplacian. Krej\v{c}i\v{r}\'{i}k--Raymond--Tu{\v{s}}ek\cite{KRT} proved that under Dirichlet boundary condition, it is influenced by eigenvalues of a Schr\"odinger operator with a potential depending on principle curvatures of the hypersurface. Krej\v{c}i\v{r}\'{i}k\cite{K} and Jimbo--Kurata\cite{JK} showed that under Dirichlet--Neumann mixed boundary condition, it is influenced by maximum values of the mean curvature of the hypersurface.

The results of Schatzman\cite{S} are closely related to our research although the geometric situation is different from this paper. Schatzman considered the asymptotic behavior of the eigenvalues of the Laplacian on thin domains with Neumann boundary condition as it degenerates to its boundary and proved that it is influenced by a geometric quantities such as the coefficients of the second fundamental form and the mean curvature of the boundary. Under our situation, Schatzman's results give us the following theorem.
\begin{thm}[Schatzman]\label{thm0}
Let $\lamkep$ be the $k$-th eigenvalue of the Laplacian in $\Omega(\e)$ with Neumann boundary condition, and let $\lamk$ be the $k$-th eigenvalue of the Laplace--Beltrami operator on $\Gamma$. Then we have 
\begin{equation*}
\lamkep = \lamk + O(\e) \,\,\, \text{as} \,\,\, \e \to 0.
\end{equation*}
Moreover if $\lamk$ is simple, then we have
\begin{equation*}
\lamkep = \lamk + O(\e^{2}) \,\,\, \text{as} \,\,\, \e \to 0.
\end{equation*}
\end{thm}

On the other hand, in the two-phase eigenvalue problem on $\Omega(\e)$, the method used in the previous study can not be applied since the coefficients are discontinuous at the interface $\Gamma$. We treat this problem by using a variational method and the Fourier expansions with respect to an appropriate orthonormal basis. This idea is based on the method used in \cite{JK}. It works well even if the coefficients are discontinuous.

Now we present the main results of this paper about the asymptotic behavior of $\lamkep$ as $\e \to 0$.

\begin{thm}\label{thm1}
Let $\lamkep$ be the $k$-th eigenvalue of the two-phase eigenvalue problem \eqref{P}, and let $\lamk$ be the $k$-th eigenvalue of the Laplace--Beltrami operator on $\Gamma$. Then we have
\begin{equation*}
\lamkep = \frac{\sig_{-}+\sig_{+}}{2} \lamk + O(\e) \,\,\, \text{as} \,\,\, \e \to 0.
\end{equation*}
\end{thm}
Note that the remainder term $O(\e)$ depends on $\sig_{-}, \sig_{+}$, and $k$.   
From Theorem \ref{thm1}, we see that the influence of the discontinuity of the coefficients on the asymptotic behavior of the eigenvalues appears as the arithmetic mean of coefficients.
The outline of the proof of Theorem \ref{thm1} is as follows: first, we derive the upper bound of eigenvalues by substituting an appropriate function composed of the eigenfunctions of the Laplace--Beltrami operator on $\Gamma$ into the Rayleigh quotient. Next, we derive the lower bound of eigenvalues by taking a certain test function in \eqref{weak} that projects the eigenfunctions onto the eigenspace corresponding to the $k$-th eigenvalues of the Laplace--Beltrami operator on $\Gamma$. Such a test function is constructed by means of the Fourier expansions with respect to an appropriate orthonormal basis. 
By the upper bound of eigenvalues, we get some estimates for the Fourier coefficients of the eigenfunctions. By using the estimates, we can obtain the lower bound of eigenvalues.
Combining the upper bound and lower bound of eigenvalues, we prove Theorem \ref{thm1}. This proof is in Section \ref{prth1}.

If we suppose that $\lamk$ is simple, we obtain more precise asymptotic behavior of $\lamkep$.
\begin{thm}\label{thm2}
Suppose that $\lamk$ is simple. Then we have 
\begin{equation*}
\lamkep = \frac{\sig_{-}+\sig_{+}}{2} \lamk + \frac{\sig_{+} - \sig_{-}}{4} \Lambda_{k}\e + o(\e)\,\,\, \text{as} \,\,\, \e \to 0, 
\end{equation*}
where 
\begin{equation*}
\Lambda_{k} = \intg \sum^{n-1}_{i,j = 1} \left( \widetilde{G}^{ij} - H g^{ij}_{0} \right) \frac{\pa \Phi_{k}}{\pa \xi_{i}} \frac{\pa \Phi_{k}}{\pa \xi_{j}} \G.
\end{equation*}
\end{thm}
\begin{rmk}
Notice that $\widetilde{G}^{ij} - H g^{ij}_{0}$ are the components of the matrix $(2W - tr(W)I) g^{-1}_{0}$, where $W$ denotes the Weingarten matrix, $I$ is the identity matrix, and $g^{-1}_{0}$ is the inverse metric matrix of $\Gamma$. Note that the eigenvalues of the Weingarten matrix $W$ are the principal curvatures of $\Gamma$ and $H = tr(W)$. As the matrix $2W - tr(W) I$ is in general not positive definite neither negative definite, it follows that the term $\Lambda_{k}$ can be either positive or negative.  
\end{rmk}
Note that the remainder term $o(\e)$ depends on $\sig_{-}, \sig_{+}$, and $k$. 
The exact meaning of the symbols of Theorem \ref{thm2} is explained in Section \ref{pre}.
The term $\Lambda_{k}$ represents the geometric shape of the interface $\Gamma$, which consists of the quantities $\widetilde{G}^{ij}$ related to the second fundamental form, mean curvature $H$, and the $k$-th normalized eigenfunction $\Phi_{k}$ on the interface $\Gamma$. We mention that a term similar to $\Lambda_{k}$ appears in Schatzman's original results\cite[Section $11$, Theorem $4$]{S}. From Theorem \ref{thm2}, we see that the influence of the geometric shape of the interface $\Gamma$ appears in the second term of the asymptotic behavior of eigenvalues. Moreover, we notice that the second term only appears when $\sig_{-} \neq \sig_{+}$. 
The outline of the proof of Theorem \ref{thm2} is as follows: by Theorem \ref{thm1} and the simplicity of $\lamk$, we get a better estimate for the Fourier coefficients than the one used in the proof of Theorem \ref{thm1}. We prove Theorem \ref{thm2} by using it. This proof is in Section \ref{simplethm}.

If $\lamk$ is not simple, it is difficult to get more precise asymptotic behavior of the eigenvalues in general. If the interface $\Gamma$ is a sphere, however, we can obtain it although the eigenvalues of the sphere are not simple. 
\begin{thm}\label{thm3}
If $\Gamma$ is $S^{n-1}(r)$ $($i.e. the $n-1$ dimensional sphere with radius $r > 0$$)$, then we have
\begin{equation}\label{eigenvalue in the ball}
\lamkep = \frac{\sig_{-}+\sig_{+}}{2} \lamk + \frac{n-3}{4r}\left(\sig_{+}-\sig_{-}\right) \lamk \e + o(\e) \,\,\, \text{as} \,\,\, \e \to 0.
\end{equation}
\end{thm}
Note that the remainder term $o(\e)$ depends on $\sig_{-}, \sig_{+}$, and $k$.
The outline of the proof of Theorem \ref{thm3} is as follows: first, we obtain the coefficients of the second fundamental form of $S^{n-1}(r)$. Then by using this and the estimate for the Fourier coefficients which we have already obtained in the proof of Theorem \ref{thm2}, we prove Theorem \ref{thm3}. This proof is in Section \ref{thmsphere}.

The following sections are organized as follows: in section \ref{pre}, we give some preliminaries needed to estimate the eigenvalues. In section \ref{prth1}, we prove Theorem \ref{thm1}. In section \ref{simplethm}, based on the results of section \ref{prth1}, we prove Theorem \ref{thm2}. We obtain more precise asymptotic behavior under the assumption that $k$-th eigenvalue $\lamk$ is simple. In section \ref{thmsphere},  we prove Theorem \ref{thm3}. We obtain a precise asymptotic behavior of the eigenvalues if the interface $\Gamma$ is a sphere.

\section{Preliminaries}\label{pre}
It is known that the $k$-th eigenvalue $\lamkep$ can be characterized by the min-max principle as in \cite{CH}\cite{EdE}:
\begin{lem}[min-max principle]
For any natural number $k \geqslant 1$, 
\begin{equation}\label{minmax}
\lambda_{k,\e} = \sup_{E \subset L^{2}(\Omega{(\e)}), \dim E \leqslant k-1} \inf \{ R_{\e}(\Phi) \,\, | \,\, \Phi \in H^{1}(\Omega(\e)), \Phi \perp E \}, 
\end{equation}
where $R_{\e}(\Phi)$ is defined by
\begin{equation}\label{Rayleigh quotient1}
R_{\e}(\Phi) = \frac{\displaystyle \int_{\Omega(\e)} \sig |\nabla \Phi|^{2} dx}{\displaystyle \int_{\Omega(\e)} |\Phi|^{2} dx}.
\end{equation}
\end{lem}
This functional $R_{\e}$ is called a Rayleigh quotient. According to the min-max principle \eqref{minmax}, it is sufficient to estimate the Rayleigh quotient \eqref{Rayleigh quotient1} for the sake of the estimate of the $k$-th eigenvalue $\lamkep$. However, it is not easy to estimate the above Rayleigh quotient because $\Omega_{-}(\e)$ and $\Omega_{+}(\e)$ are perturbed as $\e \to 0$. Thus we will consider to fix the domains by a coordinate transformation. 

Since the interface $\Gamma$ is an $n-1$ dimensional compact manifold in $\Rn$, we can take the union of a finite number of local patches in $\Gamma$, each of which has local coordinates $(\xi_{1}, \xi_{2}, \cdots , \xi_{n-1})$. Note that we regard a point $\xi \in \Gamma$ as its corresponding local coordinate $(\xi_{1}, \xi_{2}, \cdots, \xi_{n-1})$ through a local coordinate map. 
Every $x \in \Omega(\e)$ in the neighborhood of the interface $\Gamma$ is represented by 
\begin{equation}\label{x}
x = \xi + t \nu_{\Gamma}(\xi), \quad \xi \in \Gamma, \,\, |t| < \e.
\end{equation}
We introduce a local coordinate $(\xi_{1}, \xi_{2}, \cdots , \xi_{n-1}, t)$ for $\Gamma \times (-\e, \e)$. Let $g = \left( g_{ij}(\xi,t) \right)$ denote the Riemannian metric associated with this local coordinate. By \eqref{x}, $g_{ij}(\xi,t)$ is given by

\begin{equation}\label{asymp1}
g_{ij}(\xi,t) = \begin{cases}
g_{0,ij}(\xi) + t \widetilde{g}_{0,ij}(\xi) + t^{2}\widehat{g}_{0,ij}(\xi) &\text{if} \quad 1\leqslant i,j \leqslant n-1, \\
0 &\text{if} \quad i = n, j \neq n \,\, \text{or}\,\, i \neq n, j = n, \\
1 &\text{if} \quad i , j = n,
\end{cases}
\end{equation}
where $g_{0} = \left( g_{0, ij}(\xi) \right)$ denote the Riemannian metric associated with the local coordinate $(\xi_{1}, \xi_{2}, \cdots , \xi_{n-1})$ and we denote
\begin{equation*}
\widetilde{g}_{0,ij} = \left(\frac{\pa}{\pa \xi_{i}}, \frac{\pa \nu_{\Gamma}}{\pa \xi_{j}}\right) + \left(\frac{\pa}{\pa \xi_{j}}, \frac{\pa \nu_{\Gamma}}{\pa \xi_{i}}\right), \quad 
\widehat{g}_{0,ij} = \left(\frac{\pa \nu_{\Gamma}}{\pa \xi_{i}}, \frac{\pa \nu_{\Gamma}}{\pa \xi_{j}}\right).
\end{equation*}
Here $\pa / \pa \xi_{i}$ and $\pa / \pa \xi_{j}$ are tangent vectors on $\xi \in \Gamma$ and $(\cdot , \cdot)$ is the Euclidean inner product.
Let $(b_{ij})_{1\leqslant i,j \leqslant n-1}$ denote the coefficients of the second fundamental form of $\Gamma$. In the local coordinate, $b_{ij} = \left(\pa^{2} / \pa \xi_{i} \pa \xi_{j}, \nu_{\Gamma} \right)$.
By the definition of $\widetilde{g}_{0,ij}$, we have
\begin{equation}\label{second}
\left(\frac{\pa}{\pa \xi_{i}}, \frac{\pa \nu_{\Gamma}}{\pa \xi_{j}}\right) = -\left( \frac{\pa^{2}}{\pa \xi_{i} \pa \xi_{j}}, \nu_{\Gamma} \right) = -b_{ij}.
\end{equation}
Therefore we obtain
\begin{equation}\label{asymp2}
\widetilde{g}_{0,ij} = -2b_{ij}.
\end{equation}

Denote the inverse matrix of $(g_{0, ij})$ by $(g^{ij}_{0})$ and $G_{0} = \det (g_{0,ij})$. Similarly, let $(g^{ij})$ denote the inverse matrix of $(g_{ij})$ and let $G = \det (g_{ij})$. 
Then by \eqref{asymp1}, we can obtain the asymptotic formulas for the inverse metric tensor $g^{ij}$ and the Jacobian $\sqrt{G}$ as follows:
\begin{align}
g^{ij}(\xi,t) &= g^{ij}_{0}(\xi) + t \widetilde{G}^{ij}(\xi) + O(t^{2}) \quad \text{as} \,\, t \to 0, \label{important formula1}\\
\sqrt{G(\xi,t)} &= \sqrt{G_{0}}\left(1 - t H(\xi)\right) + O(t^{2}) \quad \text{as} \,\, t \to 0, \label{important formula2}
\end{align}
where $\widetilde{G}^{ij}(\xi) = 2 \sum^{n-1}_{k,l = 1}g^{ik}_{0}(\xi) b_{kl}(\xi) g^{lj}_{0}(\xi)$, and $H(\xi)$ is the mean curvature of $\Gamma$ at $\xi \in \Gamma$ with respect to $\nu_{\Gamma}$ (defined as the sum of the principle curvatures of $\Gamma$). This asymptotic formulas \eqref{important formula1} and \eqref{important formula2}, which are given by Schatzman's paper\cite[Section $10$]{S}, will play an important role to obtain the asymptotic behavior of the eigenvalues.

By using the local coordinate $(\xi_{1}, \xi_{2}, \cdots , \xi_{n-1}, t)$, the norm of the gradient of $\Phi$ in \eqref{Rayleigh quotient1} is 
\begin{equation*}
|\nabla \Phi|^{2} = |\nabla_{g} \Phi|^{2} + \left(\frac{\partial \Phi}{\partial t}\right)^{2},
\end{equation*}
where
\begin{equation*}
|\nabla_{g} \Phi|^{2} = \sum^{n-1} _{i,j=1} g^{ij} \frac{\pa \Phi}{\pa \xi_{i}} \frac{\pa \Phi}{\pa \xi_{j}}.
\end{equation*}
Similarly we define 
\begin{equation*}
|\nabla_{g_{0}} \Phi|^{2} = \sum^{n-1} _{i,j=1} g^{ij}_{0} \frac{\pa \Phi}{\pa \xi_{i}} \frac{\pa \Phi}{\pa \xi_{j}}.
\end{equation*}
In terms of the local coordinate $(\xi_{1}, \xi_{2}, \cdots , \xi_{n-1}, t)$, the Rayleigh quotient reads
\begin{equation}\label{ray}
R_{\e}(\Phi) = \frac{\displaystyle \int_{-\e}^{\e}\int_{\Gamma} \sig \left( |\nabla_{g} \Phi|^{2} + \left( \frac{\partial \Phi}{\partial t} \right)^{2} \right) \sqrt{G(\xi,t)} d\xi dt}{\displaystyle \int_{-\e}^{\e}\int_{\Gamma} |\Phi|^{2}\sqrt{G(\xi,t)} d\xi dt}.
\end{equation}
By introducing the variable $\tau \in (-1, 1)$ by $t = \e \tau$ and transform $\widetilde{\Phi}_{\e}(\xi,\tau) = \Phi(\xi,\e \tau)$, we rewrite the min-max principle and the Rayleigh quotient as follows:
\begin{equation*}
\lambda_{k,\e} = \sup_{E \subset L^{2}(\Gamma \times (-1,1)),\dim E \leqslant k-1} \inf \{ R_{\e}(\widetilde{\Phi}_{\e}) \,\, | \,\, \widetilde{\Phi}_{\e} \in H^{1}(\Gamma \times (-1,1)), \widetilde{\Phi}_{\e} \perp_{\e} E \},
\end{equation*}
where 
\begin{equation}\label{ray2}
R_{\e}(\widetilde{\Phi}_{\e}) = \frac{\displaystyle \inta \sig \left( |\nabla_{g} \widetilde{\Phi}_{\e}|^{2} + \frac{1}{\e^{2}}\left( \frac{\pa \widetilde{\Phi}_{\e}}{\pa \tau} \right)^{2} \right) \sqrt{G(\xi,\e \tau)} d\xi d\tau}{\displaystyle \inta |\widetilde{\Phi}_{\e}|^{2} \sqrt{G(\xi,\e \tau)} d\xi d\tau}
\end{equation}
and $\widetilde{\Phi}_{\e} \perp_{\e} E$ means that for any $\Psi \in E$,
\begin{equation*}
\inta \widetilde{\Phi}_{\e}(\xi,\tau) \Psi(\xi,\tau) \sqrt{G(\xi,\e \tau)} d\xi d\tau = 0.
\end{equation*}

We will estimate the eigenvalue $\lamkep$ by using the Rayleigh quotient \eqref{ray2}.
In the following sections, we will denote $\sqrt{G(\xi,\e \tau)}$ by $\sqrt{G_{\e}}$ for simplicity and by $C$ a positive constant independent of $\e$. The same letter $C$ will be used to denote different constants.
\section{Proof of Theorem \bf{\ref{thm1}}}\label{prth1}
\subsection{Upper estimate of eigenvalues}
Let $k \geqslant 1$ and $E \subset L^{2}(\Gamma \times (-1,1))$ be any subspace such that $\dim E \leqslant k-1$. We define
\begin{equation*}
L_{k} = \text{Span} \left[ \frac{1}{\sqrt{2}}\Phi_{1}(\xi),\frac{1}{\sqrt{2}}\Phi_{2}(\xi), \ldots \frac{1}{\sqrt{2}}\Phi_{k}(\xi) \right], 
\end{equation*}
where $\Phi_{k}$ is the normalized eigenfunction associated with the $k$-th eigenvalue $\lamk$ of the Laplace--Beltrami operator on $\Gamma$. 
Since $\dim E \leqslant k-1$ and $\dim L_{k} = k$, there exist constants $\{ c_{p}(\e) \}^{k}_{p=1}$ such that
\begin{equation*}
\widetilde{\Psi}_{\e}(\xi,\tau) = \sum_{p=1}^{k} c_{p}(\e) \frac{1}{\sqrt{2}}\Phi_{p}(\xi), \quad \widetilde{\Psi}_{\e} \perp_{\e} E. 
\end{equation*}
We substitute $\widetilde{\Psi}_{\e}$ into the Rayleigh quotient \eqref{ray2} as a test function. Then we have 
\begin{align*}
&\inf \{ R_{\e}(\widetilde{\Phi}_{\e}) \,\, | \,\, \widetilde{\Phi}_{\e} \in H^{1}(\Gamma \times (-1,1)), \widetilde{\Phi}_{\e} \perp_{\e} E \} \\
\leqslant R_{\e}(\widetilde{\Psi}_{\e}) 
&= \dfrac{\sig_{-} \displaystyle \intm |\nabla_{g} \widetilde{\Psi}_{\e}|^{2} \Gep
+ \sig_{+} \displaystyle \intp |\nabla_{g} \widetilde{\Psi}_{\e}|^{2} \Gep }{\displaystyle \inta |\widetilde{\Psi}_{\e}|^{2} \Gep} =: \dfrac{N_{1}(\e) + N_{2}(\e)}{M(\e)}, 
\end{align*}
where 
\begin{align*}
N_{1}(\e) &= \sig_{-} \intm |\nabla_{g} \widetilde{\Psi}_{\e}|^{2}\Gep, \\
N_{2}(\e) &= \sig_{+} \intp |\nabla_{g} \widetilde{\Psi}_{\e}|^{2}\Gep, \\
M(\e) &= \inta |\widetilde{\Psi}_{\e}|^{2}\Gep. 
\end{align*}
By requiring the normalization, 
\begin{equation*}
M(\e) = \inta |\widetilde{\Psi}_{\e}|^{2}\Gep = 1.
\end{equation*}
We obtain 
\begin{equation}\label{coefficient}
\sum^{k}_{p=1}c_{p}(\e)^2 = 1 + O(\e).
\end{equation}
By using the asymptotic formulas for the inverse metric tensor \eqref{important formula1} and the Jacobian \eqref{important formula2}, we calculate both terms $N_{1}(\e)$ and $N_{2}(\e)$. 
\begin{align*}
N_{1}(\e) &= \sig_{-} \intm |\nabla_{g} \widetilde{\Psi}_{\e}|^{2} \Gep \\
&= \frac{\sig_{-}}{2} \sum^{k}_{p,q=1} c_{p}(\e) c_{q}(\e) \intm \nabla_{g_{0}} \Phi_{p} \cdot \nabla_{g_{0}} \Phi_{q} \sqrt{G_{0}}d\xi d\tau + O(\e) \\
&= \frac{\sig_{-}}{2} \sum^{k}_{p=1} \lambda_{p} c_{p}(\e)^2 + O(\e). 
\end{align*}
Similary, we have 
\begin{align*}
N_{2}(\e) &= \sig_{+} \intp |\nabla_{g} \widetilde{\Psi}_{\e}|^{2} \Gep \\
&= \frac{\sig_{+}}{2} \sum^{k}_{p=1} \lambda_{p} c_{p}(\e)^2 + O(\e). 
\end{align*}
Thus we have
\begin{align*}
\dfrac{N_{1}(\e) + N_{2}(\e)}{M(\e)} &= \frac{\sig_{-}}{2} \sum^{k}_{p=1} \lambda_{p} c_{p}(\e)^2 + \frac{\sig_{+}}{2} \sum^{k}_{p=1} \lambda_{p} c_{p}(\e)^2 + C\e \\
&= \dfrac{\sig_{-}+\sig_{+}}{2} \sum^{k}_{p=1} \lambda_{p} c_{p}(\e)^2 + C\e \\
&\leqslant \dfrac{\sig_{-}+\sig_{+}}{2} \lamk \sum^{k}_{p=1}c_{p}(\e)^2 + C\e \\
&= \dfrac{\sig_{-}+\sig_{+}}{2} \lamk + C\e.
\end{align*}
Here we used the monotonicity of the eigenvalues $\{ \lamk \}_{k\geqslant1}$ and \eqref{coefficient}. Therefore we obtain the following upper estimate of the eigenvalue $\lamkep$. 
\begin{equation}\label{upper estimate}
\lamkep \leqslant \dfrac{\sig_{-}+\sig_{+}}{2} \lamk + C\e.
\end{equation}

\subsection{Lower estimate of eigenvalues}
For any $\psi \in H^{1}(\Gamma \times (-1,1))$, we consider the weak form of \eqref{P} in the local coordinate:
\begin{equation}\label{eigenequation}
\inta \sig \left( \nabla_{g} \efunc \, \cdot \, \nabla_{g} \psi + \frac{1}{\e^{2}} \frac{\pa \efunc}{\pa \tau} \frac{\pa \psi}{\pa \tau} \right) \Gep
= \lamkep \inta \efunc \psi \Gep,  
\end{equation}
where $\lamkep$ is the $k$-th eigenvalue and $\efunc$ is the $k$-th eigenfunction associated with $\lamkep$. We normalize $\efunc$ as follows:
\begin{equation}\label{normalize}
\inta |\efunc|^{2} \Gep = 1.
\end{equation}
If we take $\psi = \efunc$, then we have 
\begin{equation}\label{eigenvalue}
\lamkep = \inta \sig \left( |\nabla_{g} \efunc|^{2} + \frac{1}{\e^{2}} \left( \frac{\pa \efunc}{\pa \tau} \right)^{2} \right) \Gep.
\end{equation}

The main idea to get the lower estimate of eigenvalues $\lamkep$ is to take a test function which projects $\efunc$ onto the eigenspace of $\lamk$. 
For that reason, we consider the Fourier expansions of $\efunc$. Let $\Phi_{p}$ $(p \geqslant 1)$ denote the $p$-th normalized eigenfunction of Laplace--Beltrami operator on $\Gamma$ and $\phi_{l}$ ($l \geqslant 1$) the $l$-th normalized eigenfunction of the following eigenvalue problem: 
\begin{equation}\label{eigenvalue problem for thickness direction}
-\frac{d^{2} \phi}{d \tau^{2}}(\tau) = \mu \phi(\tau) \quad (\tau \in (-1, 1) ), \quad \phi'(-1) = 0, \quad \phi'(1) = 0. 
\end{equation}
Note that the eigenvalue problem \eqref{eigenvalue problem for thickness direction} can be solved explicitly as follows:
\begin{equation}\label{eigenvalue and eigenfunction for thickness direction}
\mu_{l}  = \frac{(l-1)^{2}}{4} \pi^{2} \,\, (l \geqslant 1), \quad 
\phi_{l} = \begin{cases}
\dfrac{1}{\sqrt{2}} &(l = 1), \\
\cos (\dfrac{l-1}{2} \pi \tau ) &(l \geqslant 2, \, l \,\, \mathrm{odd}), \\
\sin (\dfrac{l-1}{2} \pi \tau ) &(l \geqslant 2, \, l \,\, \mathrm{even}).
\end{cases}
\end{equation}
Then, since $\left\{ \Phi_{p} \phi_{l} \right\}_{p \geqslant 1, l \geqslant 1}$ is an orthonormal basis of $L^{2}\left(\Gamma \times (-1,1) \right)$, the $k$-th eigenfunction $\efunc$ admits the following the Fourier expansions: 
\begin{align}\label{Fourier expansion}
\efunc(\xi, \tau) &= \sum^{\infty}_{p,l=1} \alpha^{p, l}_{k} (\e) \Phi_{p}(\xi) \phi_{l}(\tau), \\
\alpha^{p, l}_{k} (\e) &= \inta \efunc(\xi, \tau)\Phi_{p}(\xi) \phi_{l}(\tau) \G d\tau.
\end{align}

First of all, we will get some estimates for the Fourier coefficients $\alpha^{p, l}_{k}(\e)$. If we substitute the Fourier expansions \eqref{Fourier expansion} into \eqref{eigenvalue}, then we can obtain some estimates for the Fourier coefficients $\alpha^{p, l}_{k}(\e)$ by using the upper bound of $\lamkep$. 
\begin{lem}
The following estimates hold:
\begin{align}
\sum^{\infty}_{p=1, l=2} \alpha^{p, l}_{k}(\e)^{2} &= O(\e^{2}) \,\,\, \text{as} \,\,\, \e \to 0, \label{fourier coefficient estimate2}\\
\sum^{\infty}_{p=1} \alpha^{p, 1}_{k}(\e)^{2} &= 1 + O(\e) \,\,\, \text{as} \,\,\, \e \to 0. \label{fourier coefficient estimate3}
\end{align}
\end{lem}
\begin{proof}
By using the upper bound of $\lamkep$, we have
\begin{equation}\label{tau estimate}
\inta \left( \frac{\pa \efunc}{\pa \tau} \right)^{2} \Gt \leqslant C\e^{2}.
\end{equation}
We substitute the Fourier expansions \eqref{Fourier expansion} into \eqref{tau estimate}, then we obtain
\begin{equation*}
\sum^{\infty}_{p=1, l=2} \mu_{l} (\alpha^{p, l}_{k})^{2} \leqslant C\e^{2}.
\end{equation*}
Note $\mu_{l}\geqslant \mu_{2} = \pi^{2}/4$ for $l \geqslant 2$. Thus we get \eqref{fourier coefficient estimate2}.
Moreover, by the normalization \eqref{normalize}, we obtain 
\begin{equation}\label{esti}
\sum^{\infty}_{p,l = 1} (\alpha^{p, l}_{k})^{2} = 1 + O(\e).
\end{equation}
Combining the estimate \eqref{fourier coefficient estimate2} with \eqref{esti}, we obtain \eqref{fourier coefficient estimate3}.
\end{proof}

Next, we will consider the lower estimate of eigenvalue $\lamkep$. As a test function of \eqref{eigenequation}, we take
\begin{equation}\label{test}
\psi = \sum^{k(j+1) -1}_{p=k(j)} \alpha^{p,1}_{k} \Phi_{p} \phi_{1},
\end{equation}
where $\{ k(j) \}^{\infty}_{j=1}$ is the increasing sequence of natural numbers defined by 
\begin{equation}\label{dd}
k(1) = 1, \quad k(j+1) = \min \{ k \in \N \,\,|\,\, \lamk > \lambda_{k(j)} \}.
\end{equation}
We note that for any $k$, there exists a unique $j$ such that $k(j) \leqslant k < k(j + 1)$.
By the definition \eqref{dd}, the multiplicity of $\lamk$ is given by $k(j+1) - k(j)$. Thus the test function \eqref{test} taken as above is to project onto the eigenspace of $\lamk$.
By using the asymptotic behavior of the inverse metric tensor \eqref{important formula1} and the Jacobian \eqref{important formula2} and also using the orthonormality of eigenfunction $\Phi_{p}$ and $\phi_{l}$, the right hand side of \eqref{eigenequation} is 
\begin{equation}\label{RHS}
\lamkep \inta \efunc \psi \Gep = \lamkep \sum^{k(j+1) -1}_{p=k(j)} (\alpha^{p,1}_{k})^{2} + O(\e).
\end{equation}
Moreover, the left hand side of \eqref{eigenequation} is 
\begin{align*}
& \inta \sig \left( \nabla_{g} \efunc \cdot \nabla_{g} \psi + \frac{1}{\e^{2}} \frac{\pa \efunc}{\pa \tau} \frac{\pa \psi}{\pa \tau} \right) \Gep \notag \\
&= \sum^{k(j+1) -1}_{p=k(j)} \sum^{\infty}_{l=1} \alpha^{p,l}_{k}\alpha^{p,1}_{k} \lambda_{p} \left( \sig_{-} \int^{0}_{-1} \phi_{l}\phi_{1} d\tau + \sig_{+} \int^{1}_{0} \phi_{l}\phi_{1} d\tau \right) + O(\e) \notag \\
&=  \frac{\sig_{-}+\sig_{+}}{2} \lamk \sum^{k(j+1) -1}_{p=k(j)} (\alpha^{p,1}_{k})^{2} + \frac{\sig_{+}-\sig_{-}}{\sqrt{2}} \lamk \sum^{k(j+1) -1}_{p=k(j)} \sum^{\infty}_{l=2} \alpha^{p,l}_{k}\alpha^{p,1}_{k} \int^{1}_{0} \phi_{l} d\tau + O(\e), 
\end{align*}
here we used $\lambda_{p} = \lamk$ $(k(j) \leqslant p \leqslant k(j+1)-1)$ and $\int^{1}_{-1} \phi_{l} d\tau = 0$ $(l \geqslant 2)$. Let us define 
\begin{equation*}
I_{1} = \frac{\sig_{+}-\sig_{-}}{\sqrt{2}} \lamk \sum^{k(j+1) -1}_{p=k(j)} \sum^{\infty}_{l=2} \alpha^{p,l}_{k}\alpha^{p,1}_{k} \int^{1}_{0} \phi_{l} d\tau. 
\end{equation*}

Now we will estimate the term $I_{1}$.  
\begin{align*}
I_{1} &= \frac{\sig_{+}-\sig_{-}}{\sqrt{2}} \lamk \sum^{k(j+1) -1}_{p=k(j)} \sum^{\infty}_{l=2} \alpha^{p,l}_{k}\alpha^{p,1}_{k} \int^{1}_{0} \phi_{l} d\tau \\
&= \frac{\sig_{+}-\sig_{-}}{\sqrt{2}} \lamk \sum^{k(j+1) -1}_{p=k(j)} \sum^{\infty}_{m=1} \left( \alpha^{p,2m}_{k}\alpha^{p,1}_{k} \int^{1}_{0} \phi_{2m} d\tau + \alpha^{p,2m+1}_{k}\alpha^{p,1}_{k} \int^{1}_{0} \phi_{2m+1} d\tau \right).
\end{align*}
By \eqref{eigenvalue and eigenfunction for thickness direction}, the following integrals can be calculated explicitly for $m \geqslant 1$: 
\begin{align*}
\int^{1}_{0} \phi_{2m} d\tau &= \frac{2}{(2m-1)\pi}, \\
\int^{1}_{0} \phi_{2m+1} d\tau &= 0.
\end{align*}
Then we have
\begin{equation*}
I_{1} = \frac{\sig_{+}-\sig_{-}}{\sqrt{2}} \lamk \sum^{k(j+1) -1}_{p=k(j)} \sum^{\infty}_{m=1} \alpha^{p,2m}_{k}\alpha^{p,1}_{k} \frac{2}{(2m-1)\pi}.
\end{equation*}
Thus we get the following estimate:
\begin{align*}
|I_{1}| &\leqslant \frac{\sqrt{2}|\sig_{+}-\sig_{-}|}{\pi}\lamk \sum^{k(j+1) -1}_{p=k(j)} \sum^{\infty}_{m=1} |\alpha^{p,2m}_{k}| |\alpha^{p,1}_{k}|\frac{1}{2m-1} \\
&\leqslant \frac{\sqrt{2}|\sig_{+}-\sig_{-}|}{\pi}\lamk \sum^{k(j+1) -1}_{p=k(j)} \sum^{\infty}_{m=1} \left( \frac{\e}{2} \cdot \frac{(\alpha^{p,1}_{k})^{2}}{(2m-1)^{2}} + \frac{1}{2\e} \cdot (\alpha^{p,2m}_{k})^{2} \right) \\
&\leqslant \frac{|\sig_{+}-\sig_{-}|}{\sqrt{2}\pi}\lamk \e \sum^{k(j+1) -1}_{p=k(j)} (\alpha^{p,1}_{k})^{2} \cdot \frac{\pi^{2}}{8} + \frac{|\sig_{+}-\sig_{-}|}{\sqrt{2}\pi \e}\lamk \sum^{\infty}_{p=1} \sum^{\infty}_{l=2} (\alpha^{p,l}_{k})^{2}, 
\end{align*} 
where we used Cauchy's inequality and the known identity $\sum^{\infty}_{m=1} 1/(2m-1)^{2} = \pi^{2}/8$. Moreover, by using the estimate for the Fourier coefficients \eqref{fourier coefficient estimate2} we have
\begin{equation*}
|I_{1}| \leqslant C\lamk \e \sum^{k(j+1) -1}_{p=k(j)} (\alpha^{p,1}_{k})^{2} + C\lamk \e. 
\end{equation*}
Thus we get the following estimate for $I_{1}$,  
\begin{equation*}
I_{1} \geqslant - C\lamk \e \sum^{k(j+1) -1}_{p=k(j)} (\alpha^{p,1}_{k})^{2} - C\lamk \e.
\end{equation*}
Therefore, we obtain the estimate of the left hand side of \eqref{eigenequation}. 
\begin{align}
&\frac{\sig_{-}+\sig_{+}}{2} \lamk \sum^{k(j+1) -1}_{p=k(j)} (\alpha^{p,1}_{k})^{2} + I_{1} -C\e \notag \\
&\geqslant \frac{\sig_{-}+\sig_{+}}{2} \lamk \sum^{k(j+1) -1}_{p=k(j)} (\alpha^{p,1}_{k})^{2} \left(1 - \frac{2C}{\sig_{-}+\sig_{+}}\e \right) - C\lamk \e -C\e. \label{LHS}
\end{align}
Combining the estimate of the right hand side of \eqref{RHS} with that of the left hand side of \eqref{LHS} yields
\begin{equation}\label{estimate1}
\lamkep \sum^{k(j+1) -1}_{p=k(j)} (\alpha^{p,1}_{k})^{2} \geqslant \frac{\sig_{-}+\sig_{+}}{2} \lamk \sum^{k(j+1) -1}_{p=k(j)} (\alpha^{p,1}_{k})^{2} \left( 1 - \frac{2C}{\sig_{-}+\sig_{+}}\e \right) - C\lamk \e - C\e.
\end{equation}
From the above estimate, it will be necessary to show the following lemma to obtain the lower estimate of eigenvalues.
\begin{lem}\label{key lemma1}
The following estimate holds:
\begin{equation}\label{lemma1}
\sum^{k(j+1) -1}_{p=k(j)} \alpha^{p, 1}_{k}(\e)^{2} = 1 + o(1) \,\,\, \text{as} \,\,\, \e \to 0.
\end{equation}
\end{lem}

\begin{proof}
We will study the limit of $\efunc$ for $\e \to 0$. Using the upper estimate of $\lamkep$, we have the boundedness of $\{ \efunc \}_{\e > 0} \subset H^{1}(\Gamma \times (-1,1))$. Applying Rellich's Theorem, we can take a subsequence $\{ \e_{p} \}^{\infty}_{p=1}$, a nonnegative value $\widehat{\lambda}_{k}$, and a function $\widehat{\Phi}_{k} \in H^{1}(\Gamma \times (-1,1))$ such that
\begin{align}\label{convergence}
\begin{cases}
\displaystyle \lim_{p \to \infty} \lambda_{k, \e_{p}} = \widehat{\lambda}_{k}, \\
\displaystyle \lim_{p \to \infty} \norm{\widetilde{\Phi}_{k,\e_{p}} - \widehat{\Phi}_{k}}_{L^{2}(\Gamma \times (-1,1))}= 0, \\
\widetilde{\Phi}_{k,\e_{p}} \rightharpoonup \widehat{\Phi}_{k} \,\, \text{weakly}\,\, \text{in} \,\, H^{1}(\Gamma \times (-1,1)).
\end{cases}
\end{align}
By using weak lower semicontinuity of $H^{1}$-norm and the estimate \eqref{tau estimate}, we show that $\widehat{\Phi}_{k}(\xi,\tau)$ is independent of the variable $\tau$. Therefore we can denote $\widehat{\Phi}_{k}(\xi,\tau) = \widehat{\Phi}_{k}(\xi)$.
If we take $\psi(\xi,\tau) = \widehat{\psi}( \xi) \,\,(\widehat{\psi} \in H^{1}(\Gamma))$ as a test function in \eqref{eigenequation} and $p \to \infty$, then by \eqref{convergence} we obtain
\begin{equation}\label{convergence eq}
\intg \nabla_{g_{0}} \widehat{\Phi}_{k} \cdot \nabla_{g_{0}} \widehat{\psi} \G
= \widetilde{\lambda}_{k} \intg \widehat{\Phi}_{k}\widehat{\psi} \G \qquad \text{for any } \widehat{\psi} \in H^{1}(\Gamma),
\end{equation}
where we put
\begin{equation*}
\widetilde{\lambda}_{k} = \frac{2\widehat{\lambda}_{k}}{\sig_{-}+\sig_{+}}.
\end{equation*}
For any test function $\widehat{\psi} \in H^{1}(\Gamma)$, \eqref{convergence eq} holds. Thus $\widetilde{\lambda}_{k}$ is an eingenvalue of the Laplace--Beltrami operator on $\Gamma$ and $\widehat{\Phi}_{k}$ is a corresponding eigenfunction. By the upper bound of $\lamkep$, we get 
\begin{equation*}
\widehat{\lambda}_{k} \leqslant \frac{\sig_{-}+\sig_{+}}{2}\lamk.
\end{equation*}
Therefore we obtain 
\begin{equation}\label{upper}
\widetilde{\lambda}_{k} \leqslant \lamk.
\end{equation}

Also by using the orthonormality of $\efunc$, for each $k, k' \geqslant 1$ we have
\begin{equation*}
\inta \widehat{\Phi}_{k} \widehat{\Phi}_{k'} \Gt = \delta(k,k').
\end{equation*}
Thus we get
\begin{equation}\label{orthonormality}
\intg \widehat{\Phi}_{k} \widehat{\Phi}_{k'} \G = \frac{\delta(k,k')}{2}.
\end{equation}
Using the orthonormality condition \eqref{orthonormality}, we have
\begin{equation}\label{lower}
\widetilde{\lambda}_{k} \geqslant \lamk.
\end{equation}
Therefore from \eqref{upper} and \eqref{lower}, we obtain
\begin{equation*}
\widetilde{\lambda}_{k} = \lamk.
\end{equation*}
This implies that $\widetilde{\lambda}_{k}$ is the $k$-th eigenvalue and $\widehat{\Phi}_{k}$ is the corresponding $k$-th eigenfunction. Thus the eigenspace of $\lamk$ contains $\widehat{\Phi}_{k}$. We can express $\widehat{\Phi}_{k}$ as
\begin{equation}\label{wilde phi}
\widehat{\Phi}_{k}(\xi) = \sum^{k(j+1)-1}_{p=k(j)} c_{p} \Phi_{p}(\xi),
\end{equation}
where $\{c_{p}\}^{k(j+1)-1}_{p=k(j)}$ are suitable constants. By \eqref{orthonormality}, we get
\begin{equation*}
\sum^{k(j+1)-1}_{p=k(j)} c^{2}_{p} = \frac{1}{2}.
\end{equation*}
Thus we have 
\begin{align*}
\norm{\widetilde{\Phi}_{k,\e} - \widehat{\Phi}_{k}}^{2}_{L^{2}(\Gamma \times (-1,1))} &= \norm{\sum^{\infty}_{p,l=1} \alpha^{p,l}_{k}\Phi_{p}\phi_{l} - \sum^{k(j+1)-1}_{p=k(j)} c_{p} \Phi_{p}}_{L^{2}(\Gamma \times (-1,1))}^{2} \\
&= \sum^{\infty}_{p,l=1} (\alpha^{p,l}_{k})^{2} - 2\sum^{k(j+1)-1}_{p=k(j)} \alpha^{p,1}_{k} \sqrt{2}c_{p} + 2 \sum^{k(j+1)-1}_{p=k(j)} c^{2}_{p} \\
&= 2\left( 1 - \sum^{k(j+1)-1}_{p=k(j)} \alpha^{p,1}_{k} \sqrt{2}c_{p} \right) + O(\e).
\end{align*}
From \eqref{convergence}, since $||\widetilde{\Phi}_{k,\e} - \widehat{\Phi}_{k}||^{2}_{L^{2}(\Gamma \times (-1,1))} \to 0$ as $\e \to 0$, we obtain
\begin{equation}\label{estimateA}
\sum^{k(j+1)-1}_{p=k(j)} \alpha^{p,1}_{k} \sqrt{2}c_{p} = 1 + o(1).
\end{equation}
By using Cauchy--Schwarz's inequality for \eqref{estimateA} and the estimate \eqref{fourier coefficient estimate3}, we have
\begin{align*}
1+o(1) = \sum^{k(j+1)-1}_{p=k(j)} \alpha^{p,1}_{k} \sqrt{2}c_{p} &\leqslant \left( \sum^{k(j+1)-1}_{p=k(j)} (\alpha^{p,1}_{k})^{2} \right)^{1/2} \left( \sum^{k(j+1)-1}_{p=k(j)} (\sqrt{2}c_{p})^{2} \right)^{1/2} \\
&= \left( \sum^{k(j+1)-1}_{p=k(j)} (\alpha^{p,1}_{k})^{2} \right)^{1/2} \\
&\leqslant \left( \sum^{\infty}_{p=1} (\alpha^{p,1}_{k})^{2} \right)^{1/2} = 1 + O(\e).
\end{align*}
Therefore we obtain
\begin{equation*}
\sum^{k(j+1)-1}_{p=k(j)} (\alpha^{p,1}_{k})^{2} = 1 + o(1).
\end{equation*}
\end{proof}

From the estimate \eqref{estimate1} and Lemma \ref{key lemma1}, we have
\begin{align*}
\lamkep &\geqslant \frac{\sig_{-}+\sig_{+}}{2} \lamk \left( 1 - \frac{2C}{\sig_{-}+\sig_{+}}\e \right) - \left( C\lamk \e + C\e \right) / \sum^{k(j+1) -1}_{p=k(j)} (\alpha^{p,1}_{k})^{2} \\
&= \frac{\sig_{-}+\sig_{+}}{2} \lamk \left( 1 - \frac{2C}{\sig_{-}+\sig_{+}}\e \right) - \frac{C\lamk \e + C\e}{1+o(1)} \\
&\geqslant \frac{\sig_{-}+\sig_{+}}{2} \lamk - C\e.
\end{align*}
Thus we obtain the following lower estimate for the eigenvalue $\lamkep$:
\begin{equation}\label{lower estimate}
\lamkep \geqslant \frac{\sig_{-}+\sig_{+}}{2} \lamk - C\e.
\end{equation}
Combining the upper estimate \eqref{upper estimate} with the lower estimate \eqref{lower estimate}, we obtain the complete proof of Theorem $\ref{thm1}$.

\section{More precise asymptotic behavior if $\lamk$ is simple}\label{simplethm}
If we suppose that the eigenvalue $\lamk$ is simple, we obtain more precise asymptotic behavior of $\lamkep$. In order to prove this, first of all we need to get a better estimate for the Fourier coefficients of $\efunc$. 
\begin{lem}\label{key lemma2}
The following estimate holds:
\begin{equation}\label{sharp estimate for Fourier coefficients}
\sum^{\infty}_{p=1, l=2} \alpha^{p, l}_{k}(\e)^{2} = o(\e^{2}) \,\,\, \text{as} \,\,\, \e \to 0.
\end{equation}
\end{lem}
\begin{proof}
We recall that $\widehat{\Phi}_{k}$ expressed by \eqref{wilde phi} satisfies $||\widetilde{\Phi}_{k,\e} - \widehat{\Phi}_{k}||^{2}_{L^{2}(\Gamma \times (-1,1))} \to 0$ as $\e \to 0$ and also \eqref{convergence eq}. In \eqref{convergence eq}, we take the same test function in \eqref{eigenequation}. Then we obtain
\begin{equation}\label{weak form 2}
\intg \nabla_{g_0} \widehat{\Phi}_{k} \cdot \nabla_{g_0} \psi \G
= \lambda_{k} \intg \widehat{\Phi}_{k}\psi \G.
\end{equation}
By \eqref{weak form 2}, we get
\begin{multline}\label{weak form 3}
\sig_{-} \intm \nabla_{g_0} \widehat{\Phi}_{k} \cdot \nabla_{g_0} \psi \G d\tau + \sig_{+} \intp \nabla_{g_0} \widehat{\Phi}_{k} \cdot \nabla_{g_0} \psi \Gt \\
= \frac{\sig_{-} + \sig_{+}}{2}\lamk \inta \widehat{\Phi}_{k} \psi \G d\tau + \frac{\sig_{-} - \sig_{+}}{2}\lamk \intg \left( \int^{0}_{-1} \psi d\tau - \int^{1}_{0} \psi d\tau \right) \widehat{\Phi}_{k} \G.
\end{multline}

Subtracting \eqref{weak form 3} from \eqref{eigenequation} and taking $\psi = \efunc - \widehat{\Phi}_{k}$, then we obtain
\begin{align*}
&\inta \sig \left( \left| \nabla_{g_0} (\efunc - \widehat{\Phi}_{k}) \right|^{2} + \frac{1}{\e^{2}} \left( \frac{\pa \efunc}{\pa \tau} \right)^{2} \right) \Gt \\
&\qquad = \left( \lamkep - \frac{\sig_{-} + \sig_{+}}{2} \lamk \right) \inta \efunc \left( \efunc - \widehat{\Phi}_{k} \right) \Gt \\
&\qquad \qquad + \frac{\sig_{-} + \sig_{+}}{2} \lamk \inta \left| \efunc - \widehat{\Phi}_{k} \right|^{2} \Gt \\
&\qquad \qquad \qquad + \frac{\sig_{+} - \sig_{-}}{2} \lamk \intg \left( \int^{0}_{-1} \efunc d\tau - \int^{1}_{0} \efunc d\tau \right)\widehat{\Phi}_{k} \G + O(\e) \\
&\qquad =: S_{1} + S_{2} + S_{3} + O(\e),
\end{align*}
where 
\begin{align}
S_{1} &= \left( \lamkep - \frac{\sig_{-} + \sig_{+}}{2} \lamk \right) \inta \efunc \left( \efunc - \widehat{\Phi}_{k} \right) \Gt, \label{first term} \\
S_{2} &= \frac{\sig_{-} + \sig_{+}}{2} \lamk \inta \left| \efunc - \widehat{\Phi}_{k} \right|^{2} \Gt, \label{second term} \\
S_{3} &= \frac{\sig_{+} - \sig_{-}}{2} \lamk \intg \left( \int^{0}_{-1} \efunc d\tau - \int^{1}_{0} \efunc d\tau \right)\widehat{\Phi}_{k} \G. \label{third term}
\end{align}
From Theorem \ref{thm1} and the fact that $||\widetilde{\Phi}_{k,\e} - \widehat{\Phi}_{k}||^{2}_{L^{2}(\Gamma \times (-1,1))} \to 0$ as $\e \to 0$, we have $S_{1} = o(\e)$ and $S_{2} = o(1)$. 
Furthermore, we can conclude that $S_{3} = O(\e)$. 
Indeed, let
\begin{equation*}
V(\xi) = \int^{0}_{-1} \efunc d\tau - \int^{1}_{0} \efunc d\tau.
\end{equation*}
Then, we get
\begin{align*}
\left| V(\xi) \right| &= \left| \int^{0}_{-1} \efunc d\tau - \int^{1}_{0} \efunc d\tau \right| \\
&\leqslant \int^{1}_{-1} \left| \int^{\tau}_{0} \frac{\pa \efunc}{\pa \tau}(\xi,s)ds \right| d\tau \\
&\leqslant 2 \left( \int^{1}_{-1} \left| \frac{\pa \efunc}{\pa \tau}(\xi,s) \right|^{2} ds \right)^{1/2}.
\end{align*}
By \eqref{tau estimate},
\begin{equation*}
\intg V(\xi)^{2} \G \leqslant 4 \inta \left| \frac{\pa \efunc}{\pa \tau}(\xi,s) \right|^{2} \G ds \leqslant C\e^{2}.
\end{equation*}
Thus we can estimate $S_{3}$ as follows:
\begin{align*}
|S_{3}| &= \left| \frac{\sig_{+} - \sig_{-}}{2} \lamk \intg V(\xi) \widehat{\Phi}_{k} \G \right| \\
&\leqslant C \left( \intg V(\xi)^{2} \G \right)^{1/2} \left( \intg \widehat{\Phi}_{k}^{2} \G \right)^{1/2} \\
&= \frac{C}{\sqrt{2}} \left( \intg V(\xi)^{2} \G \right)^{1/2} \leqslant C \e.
\end{align*}
Therefore we obtain $S_{3} = O(\e)$. 

By the above estimate, we have
\begin{multline}\label{key estimate}
\inta \sig \left( \left| \nabla_{g_0} (\efunc - \widehat{\Phi}_{k}) \right|^{2} + \frac{1}{\e^{2}} \left( \frac{\pa \efunc}{\pa \tau} \right)^{2} \right) \Gt \\
= S_{1} + S_{2} + S_{3} + O(\e) = o(1).
\end{multline}
From \eqref{key estimate}, we obtain
\begin{align}
\inta \left| \nabla_{g_0} (\efunc - \widehat{\Phi}_{k}) \right|^{2} \Gt &= o(1), \label{desired estimate1} \\
\inta \left( \frac{\pa \efunc}{\pa \tau} \right)^{2} \Gt &= o(\e^{2}). \label{desired estimate2}
\end{align}
The estimate \eqref{desired estimate2} implies the estimate for the Fourier coefficients we wanted. This completes the proof of Lemma \ref{key lemma2}.
\end{proof}
\begin{rmk}
\normalfont
Lemma \ref{key lemma2} also holds if $\lamk$ is not simple. This estimate will be used for the proof of Theorem \ref{thm3}.
\end{rmk}
Set $V = \{(p,l) \in \N^{2} \,|\, (p,l) \neq (k,1) \}$. Then we can also obtain an estimate for Fourier coefficients as follows by using the assumption that $\lamk$ is simple. 
\begin{lem}\label{key lemma3}
If $\lamk$ is simple, then we obtain the following estimate for the Fourier coefficients of $\efunc$:
\begin{equation}\label{key estimate2}
\sum_{(p,l) \in V} \alpha^{p, l}_{k}(\e)^{2}\lambda_{p} = o(1) \,\,\, \text{as} \,\,\, \e \to 0.
\end{equation}
\end{lem}
\begin{proof}
Since $\lamk$ is simple, the eigenspace of $\lamk$ is one dimensional. Thus by \eqref{wilde phi}, we have 
\begin{equation}\label{convergence if simple}
\widehat{\Phi}_{k} = \pm \frac{1}{\sqrt{2}}\Phi_{k} = \pm \Phi_{k}\phi_{1}.
\end{equation}
Without loss of generality, we may assume $\widehat{\Phi}_{k} =  \Phi_{k}\phi_{1}$. Then we have
\begin{align*}
&\inta \left| \nabla_{g_0} (\efunc - \Phi_{k}\phi_{1}) \right|^{2} \Gt \\
&= \sum_{(p_{1},l_{1}) \in V}\sum_{(p_{2},l_{2}) \in V}\alpha^{p_{1},l_{1}}_{k}\alpha^{p_{2},l_{2}}_{k} \left( \intg \nabla_{g_0} \Phi_{p_{1}} \cdot \nabla_{g_0} \Phi_{p_{2}} \G \right) \times \left( \int^{1}_{-1} \phi_{l_{1}}\phi_{l_{2}} d\tau \right) \\
&+ 2\sum_{(p,l) \in V}\alpha^{p,l}_{k}(\alpha^{k,1}_{k}-1) \left( \intg \nabla_{g_0} \Phi_{p} \cdot \nabla_{g_0} \Phi_{k} \G \right) \times \left( \int^{1}_{-1} \phi_{l}\phi_{1} d\tau \right) \\
&+ (\alpha^{k,1}_{k}-1)^{2} \left( \intg \nabla_{g_0} \Phi_{k} \cdot \nabla_{g_0} \Phi_{k} \G \right) \times \left( \int^{1}_{-1} \phi_{1}^{2} d\tau \right) \\
&= \sum_{(p,l) \in V}(\alpha^{p,l}_{k})^{2}\lambda_{p} + (\alpha^{k,1}_{k}-1)^{2}\lamk. 
\end{align*}
By \eqref{desired estimate1}, the left hand side goes to $0$ as $\e \to 0$. Therefore we obtain \eqref{key estimate2}.
\end{proof}

We take $\psi = \alpha^{k,1}_{k} \Phi_{k} \phi_{1}$ as a test function in \eqref{eigenequation}. Then the right hand side of \eqref{eigenequation} is 
\begin{align*}
&\lamkep \inta \efunc \psi \Gep \\
&= \lamkep (\alpha^{k,1}_{k})^{2} - \e \frac{\lamkep}{\sqrt{2}} \sum^{\infty}_{p=1} \sum^{\infty}_{l=2} \alpha^{p,l}_{k} \alpha^{k,1}_{k} \left( \intg H \Phi_{p} \Phi_{k} \G \right) \times \left( \int^{1}_{-1} \tau \phi_{l} d\tau \right) + O(\e^{2}) \\
&= \lamkep (\alpha^{k,1}_{k})^{2} - I_{2}\e + O(\e^{2}),
\end{align*}
where 
\begin{equation*}
I_{2} = \frac{\lamkep}{\sqrt{2}} \sum^{\infty}_{p=1} \sum^{\infty}_{l=2} \alpha^{p,l}_{k} \alpha^{k,1}_{k} \left( \intg H \Phi_{p} \Phi_{k} \G \right) \times \left( \int^{1}_{-1} \tau \phi_{l} d\tau \right).
\end{equation*}

We will estimate $I_{2}$. Set
\begin{equation*}
f_{l} = \int^{1}_{-1} \tau \phi_{l} d\tau 
\end{equation*}
and 
\begin{equation*}
F_{l} = \sum^{\infty}_{p=1} \alpha^{p,l}_{k} \Phi_{p}.
\end{equation*}
Since the eigenfunction of $\phi_{l}$ is explicitly given by \eqref{eigenvalue and eigenfunction for thickness direction}, by a forward computation we get
\begin{equation}\label{aaaa}
\sum^{\infty}_{l = 1} \left| f_{l} \right|^{2} < \infty.
\end{equation}
Furthermore, Since $H = H(\xi)$ is continuous on $\Gamma$ and $\Gamma$ is compact, we have
\begin{align*}
|I_{2}| &\leqslant C \sum^{\infty}_{l=2} |\alpha^{k,1}_{k}| \left| f_{l} \right| \left| \intg H F_{l} \Phi_{k} \G \right| \\
&\leqslant C \sum^{\infty}_{l=2} |\alpha^{k,1}_{k}| \left| f_{l} \right| \norm{F_{l}}_{L^{2}(\Gamma)}.
\end{align*}
Here $\norm{F_{l}}_{L^{2}(\Gamma)} = \left( \sum^{\infty}_{p=1} (\alpha^{p,l}_{k})^{2} \right)^{1/2}$. Thus by using Lemma \ref{key lemma2}, we have 
\begin{align*}
|I_{2}| &\leqslant C \sum^{\infty}_{l=2} |\alpha^{k,1}_{k}| \left| f_{l} \right| \left( \sum^{\infty}_{p=1} (\alpha^{p,l}_{k})^{2} \right)^{1/2} \\
&\leqslant C \left( \sum^{\infty}_{l=2} (\alpha^{k,1}_{k})^{2} |f_{l}|^{2} \right)^{1/2} \left( \sum^{\infty}_{p=1, l=2} (\alpha^{p,l}_{k})^{2} \right)^{1/2} \\
&= o(\e).
\end{align*}
Therefore, the right hand side of \eqref{eigenequation} is 
\begin{equation*}
\lamkep (\alpha^{k,1}_{k})^{2} - \e \times o(\e) + O(\e^{2}) = \lamkep (\alpha^{k,1}_{k})^{2} + O(\e^{2}). 
\end{equation*}

Moreover, by using the asymptotic formulas for the inverse metric tensor \eqref{important formula1} and the Jacobian \eqref{important formula2}, the left hand side of \eqref{eigenequation} is
\begin{align*}
&\inta \sig \left( \nabla_{g} \efunc \cdot \nabla_{g} \psi + \frac{1}{\e^{2}} \frac{\pa \efunc}{\pa \tau} \frac{\pa \psi}{\pa \tau} \right) \Gep \notag \\
&= \frac{\sig_{-}+\sig_{+}}{2}\lamk (\alpha^{k,1}_{k})^{2} + \frac{\sig_{+} - \sig_{-}}{\sqrt{2}} \lamk \sum^{\infty}_{l=2} \alpha^{k,l}_{k} \alpha^{k,1}_{k} \int^{1}_{0} \phi_{l}d\tau \\
&+ \frac{\sig_{+} - \sig_{-}}{4}\e(\alpha^{k,1}_{k})^{2}\intg \sum^{n-1}_{i,j = 1} \left( \widetilde{G}^{ij} - H g^{ij}_{0} \right) \frac{\pa \Phi_{k}}{\pa \xi_{i}} \frac{\pa \Phi_{k}}{\pa \xi_{j}} \G \\
&+ \e \sum_{(p,l) \in V} \alpha^{p,l}_{k}\alpha^{k,1}_{k} \left( \intg \sum^{n-1}_{i,j = 1} \left( \widetilde{G}^{ij} - H g^{ij}_{0} \right)\frac{\pa \Phi_{p}}{\pa \xi_{i}} \frac{\pa \Phi_{k}}{\pa \xi_{j}} \G \right) \times \left( \frac{\sig_{-}}{\sqrt{2}} \int^{0}_{-1} \tau \phi_{l} d\tau + \frac{\sig_{+}}{\sqrt{2}} \int^{1}_{0} \tau \phi_{l} d\tau \right)\\
& + O(\e^{2}).
\end{align*}

In order to prove this theorem, it suffices to show that the second and fourth term in the above are $o(\e)$. Indeed, if it were true, then by using the estimate for the Fourier coefficients $(\alpha^{k,1}_{k})^{2} = 1 + o(1)$, we prove Theorem \ref{thm2}. We can easily show that the second term is $o(\e)$ from Lemma \ref{key lemma2}. Thus if we prove that the fourth term is $o(\e)$, then it will complete the proof of this theorem. From now on, we will estimate the fourth term. 

Set
\begin{equation*}
h_{l} = \frac{\sig_{-}}{\sqrt{2}} \int^{0}_{-1} \tau \phi_{l} d\tau + \frac{\sig_{+}}{\sqrt{2}} \int^{1}_{0} \tau \phi_{l} d\tau
\end{equation*}
and 
\begin{equation*}
I_{3} = \sum_{(p,l) \in V} \alpha^{p,l}_{k}\alpha^{k,1}_{k} \left( \intg \sum^{n-1}_{i,j = 1} \left( \widetilde{G}^{ij} - H g^{ij}_{0} \right)\frac{\pa \Phi_{p}}{\pa \xi_{i}} \frac{\pa \Phi_{k}}{\pa \xi_{j}} \G \right) h_{l}.
\end{equation*}
In the same way \eqref{aaaa}, we get
\begin{equation*}
\sum^{\infty}_{l = 1} \left| h_{l} \right|^{2} < \infty.
\end{equation*}
We separate $I_{3}$ into two parts as follows:
\begin{align}
I_{3} = \sum^{\infty}_{l = 2} \alpha^{k,l}_{k}\alpha^{k,1}_{k}\left( \intg \sum^{n-1}_{i,j = 1} \left( \widetilde{G}^{ij} - H g^{ij}_{0} \right)\frac{\pa \Phi_{k}}{\pa \xi_{i}} \frac{\pa \Phi_{k}}{\pa \xi_{j}} \G \right)h_{l} \notag \\
+ \sum^{\infty}_{l = 1}\sum_{p \neq k} \alpha^{p,l}_{k}\alpha^{k,1}_{k} \left( \intg \sum^{n-1}_{i,j = 1} \left( \widetilde{G}^{ij} - H g^{ij}_{0} \right)\frac{\pa \Phi_{p}}{\pa \xi_{i}} \frac{\pa \Phi_{k}}{\pa \xi_{j}} \G \right)h_{l} \notag \\
=: A_{1} + A_{2}.
\end{align}

At first, we estimate $A_{1}$. Since the term $\widetilde{G}^{ij} - H g^{ij}_{0}$ is continuous on $\Gamma$ and $\Gamma$ is compact, we can show that there exists a constant $C > 0$ such that
\begin{equation*}
\left|\sum^{n-1}_{i,j = 1} \left( \widetilde{G}^{ij} - H g^{ij}_{0} \right)\frac{\pa \Phi_{k}}{\pa \xi_{i}}\frac{\pa \Phi_{k}}{\pa \xi_{j}} \right| \leqslant C |\nabla_{g_{0}} \Phi_{k}|^{2}.
\end{equation*}
Thus we obtain 
\begin{equation*}
\left| \intg \sum^{n-1}_{i,j = 1} \left( \widetilde{G}^{ij} - H g^{ij}_{0} \right)\frac{\pa \Phi_{k}}{\pa \xi_{i}} \frac{\pa \Phi_{k}}{\pa \xi_{j}} \G \right| \leqslant C\lambda_{k}.
\end{equation*}
Therefore,
\begin{align*}
|A_{1}| &\leqslant C\lamk \sum^{\infty}_{l = 2} |\alpha^{k,l}_{k}| |\alpha^{k,1}_{k}| |h_{l}| \\
&\leqslant C\lamk \left( \sum^{\infty}_{l = 2} (\alpha^{k,l}_{k})^{2} \right)^{1/2} \left( \sum^{\infty}_{l = 2} (\alpha^{k,1}_{k})^{2} |h_{l}|^{2} \right)^{1/2} \\
&\leqslant C\lamk \left( \sum^{\infty}_{p = 1, l = 2} (\alpha^{p,l}_{k})^{2} \right)^{1/2} \left( (\alpha^{k,1}_{k})^{2} \sum^{\infty}_{l = 2} |h_{l}|^{2} \right)^{1/2} = o(\e),
\end{align*}
where we used Lemma \ref{key lemma2}. 

Next, we estimate $A_{2}$. Let
\begin{equation*}
F_{l} = \sum_{p \neq k} \alpha^{p,l}_{k} \Phi_{p}.
\end{equation*}
Then,
\begin{align*}
|A_{2}| &\leqslant \left| \sum^{\infty}_{l = 1} \alpha^{k,1}_{k} \left( \intg \sum^{n-1}_{i,j = 1} \left( \widetilde{G}^{ij} - H g^{ij}_{0} \right)\frac{\pa F_{l}}{\pa \xi_{i}} \frac{\pa \Phi_{k}}{\pa \xi_{j}} \G \right)h_{l} \right| \notag \\
&\leqslant C \sum^{\infty}_{l = 1} |\alpha^{k,1}_{k}|  |h_{l}| \norm{\nabla_{g_{0}} F_{l}}_{L^{2}(\Gamma)} \norm{\nabla_{g_{0}} \Phi_{k}}_{L^{2}(\Gamma)}. 
\end{align*}
Here $\norm{\nabla_{g_{0}} F_{l}}_{L^{2}(\Gamma)} = \left( \sum_{p \neq k} (\alpha^{p,l}_{k})^{2} \lambda_{p} \right)^{1/2}$ and $\norm{\nabla_{g_{0}} \Phi_{k}}_{L^{2}(\Gamma)} = \lamk^{1/2}$. We get 
\begin{align*}
|A_{2}| &\leqslant C \sum^{\infty}_{l = 1} |\alpha^{k,1}_{k}|  |h_{l}| \lamk^{1/2} \left( \sum_{p \neq k} (\alpha^{p,l}_{k})^{2} \lambda_{p} \right)^{1/2} \\
&\leqslant C\left( \sum^{\infty}_{l = 1} (\alpha^{k,1}_{k})^{2} \lamk |h_{l}|^{2} \right)^{1/2} \left( \sum^{\infty}_{l = 1} \sum_{p \neq k} (\alpha^{p,l}_{k})^{2} \lambda_{p} \right)^{1/2} \\
&\leqslant C\left( (\alpha^{k,1}_{k})^{2} \lamk \sum^{\infty}_{l = 1} |h_{l}|^{2} \right)^{1/2} \left( \sum_{(p,l) \in V} (\alpha^{p,l}_{k})^{2} \lambda_{p} \right)^{1/2} = o(1).
\end{align*}
Therefore we can obtain the estimate $|I_{3}| \leqslant |A_{1}| + |A_{2}| = o(\e) + o(1) = o(1)$. This implies that the fourth term is $o(\e)$, which proves Theorem \ref{thm2}. 

\section{Asymptotic behavior of eigenvalues if the interface is a sphere}\label{thmsphere}
If $\lamk$ is not simple, it is difficult to get more precise asymptotic behavior of the eigenvalues in general. If the interface $\Gamma$ is a sphere, however, we can obtain it although the eigenvalues of the sphere are not simple. The idea of the proof is to take $\psi = \sum^{k(j+1) -1}_{p=k(j)} \alpha^{p,1}_{k} \Phi_{p} \phi_{1}$ as the test function in \eqref{eigenequation} and use a property of the coefficients of the second fundamental form of the sphere. Now, we denote by $S^{n-1}(r)$ the $n-1$ dimensional sphere with radius $r > 0$. First of all, we will consider the coefficients of the second fundamental form of $S^{n-1}(r)$. 

\begin{lem}\label{secondq sphere}
If $\Gamma = S^{n-1}(r)$, then $b_{ij} = (-1/r) g_{0,ij}$.
\end{lem}
\begin{proof}
Since $\Gamma = S^{n-1}(r)$, the outward normal vector $\nu_{\Gamma}$ to $\Gamma$ is represented by $\xi / r$. By \eqref{second}, we have
\begin{equation}\label{second quantity}
b_{ij} = -\left(\frac{\pa}{\pa \xi_{i}}, \frac{\pa \nu_{\Gamma}}{\pa \xi_{j}}\right) = -\frac{1}{r} \left( \frac{\pa}{\pa \xi_{i}}, \frac{\pa }{\pa \xi_{j}} \right) = -\frac{1}{r} g_{0,ij}. 
\end{equation}
\end{proof}
Next, we take $\psi = \sum^{k(j+1) -1}_{p=k(j)} \alpha^{p,1}_{k} \Phi_{p} \phi_{1}$ as a the test function in \eqref{eigenequation}. Then the right hand side of \eqref{eigenequation} is 
\begin{align*}
\lamkep \inta \efunc \psi \Gep = \lamkep \sum^{k(j+1)}_{p=k(j)} (\alpha^{p,1}_{k})^{2} - I_{4}\e + O(\e^{2}),
\end{align*}
where 
\begin{equation*}
I_{4} = \frac{\lamkep}{\sqrt{2}} \sum^{\infty}_{p_{1}=1} \sum^{k(j+1)}_{p_{2}=k(j)}\sum^{\infty}_{l=2} \alpha^{p_{1},l}_{k} \alpha^{p_{2},1}_{k} f_{l} \left( \intg H \Phi_{p_{1}} \Phi_{p_{2}} \G \right).
\end{equation*}
Since $\Gamma = S^{n-1}(r)$, the mean curvature $H$ (defined as the sum of the principle curvatures) equals $- (n-1)/r$. Thus we have
\begin{align*}
I_{4} &= - \frac{n-1}{r} \frac{\lamkep}{\sqrt{2}} \sum^{\infty}_{p_{1}=1} \sum^{k(j+1)}_{p_{2}=k(j)}\sum^{\infty}_{l=2} \alpha^{p_{1},l}_{k} \alpha^{p_{2},1}_{k} f_{l} \left( \intg \Phi_{p_{1}} \Phi_{p_{2}} \G \right) \\
&= - \frac{n-1}{r} \frac{\lamkep}{\sqrt{2}} \sum^{k(j+1)}_{p=k(j)} \sum^{\infty}_{l=2} \alpha^{p,l}_{k} \alpha^{p,1}_{k} f_{l}.
\end{align*}
Therefore, 
\begin{align*}
|I_{4}| \leqslant C \left( \sum^{k(j+1)}_{p=k(j)} \sum^{\infty}_{l=2} (\alpha^{p,l}_{k})^{2} \right)^{1/2} \left( \sum^{k(j+1)}_{p=k(j)} \sum^{\infty}_{l=2} (\alpha^{p,1}_{k})^{2}  f_{l}^{2} \right)^{1/2} \\
\leqslant C \left( \sum^{\infty}_{p=1} \sum^{\infty}_{l=2} (\alpha^{p,l}_{k})^{2} \right)^{1/2} \left( \sum^{k(j+1)}_{p=k(j)} (\alpha^{p,1}_{k})^{2} \sum^{\infty}_{l=2} f_{l}^{2} \right)^{1/2} = o(\e).
\end{align*}

Furthermore, the left hand side of \eqref{eigenequation} is
\begin{align*}
&\inta \sig \left( \nabla_{g} \efunc \cdot \nabla_{g} \psi + \frac{1}{\e^{2}} \frac{\pa \efunc}{\pa \tau} \frac{\pa \psi}{\pa \tau} \right) \Gep \notag \\
&= \frac{\sig_{-}+\sig_{+}}{2}\lamk \sum^{k(j+1)}_{p=k(j)} (\alpha^{p,1}_{k})^{2} + \frac{\sig_{+} - \sig_{-}}{\sqrt{2}} \lamk \sum^{k(j+1)}_{p=k(j)} \sum^{\infty}_{l=2} \alpha^{p,l}_{k} \alpha^{p,1}_{k} \int^{1}_{0} \phi_{l}d\tau \\
&+ \e \sum^{\infty}_{p_{1}, l=1} \sum^{k(j+1)}_{p_{2}=k(j)} \alpha^{p_{1},l}_{k}\alpha^{p_{2},1}_{k} h_{l} \left( \intg \sum^{n-1}_{i,j = 1} \left( \widetilde{G}^{ij} - H g^{ij}_{0} \right)\frac{\pa \Phi_{p_{1}}}{\pa \xi_{i}} \frac{\pa \Phi_{p_{2}}}{\pa \xi_{j}} \G \right) + O(\e^{2}).
\end{align*}
The second term can be treated in the same way as $I_{4}$: we can show that it is $o(\e)$. Moreover, by the definition of the term $\widetilde{G}^{ij}$ and Lemma \ref{secondq sphere} we have
\begin{align*}
\widetilde{G}^{ij} - H g^{ij}_{0} &= 2 \sum^{n-1}_{k,l = 1}g^{ik}_{0} b_{kl} g^{lj}_{0} - H g^{ij}_{0} \\
&= -\frac{2}{r} \sum^{n-1}_{k,l = 1}g^{ik}_{0} g_{0,kl} g^{lj}_{0} + \frac{n-1}{r} g^{ij}_{0} \\
&= -\frac{2}{r} g^{ij}_{0} + \frac{n-1}{r} g^{ij}_{0} = \frac{n-3}{r} g^{ij}_{0}.
\end{align*}
Thus,
\begin{align*}
&\sum^{\infty}_{p_{1}, l=1} \sum^{k(j+1)}_{p_{2}=k(j)} \alpha^{p_{1},l}_{k}\alpha^{p_{2},1}_{k} h_{l} \left( \intg \sum^{n-1}_{i,j = 1} \left( \widetilde{G}^{ij} - H g^{ij}_{0} \right)\frac{\pa \Phi_{p_{1}}}{\pa \xi_{i}} \frac{\pa \Phi_{p_{2}}}{\pa \xi_{j}} \G \right) \\
&= \frac{n-3}{r}\sum^{\infty}_{p_{1}, l=1} \sum^{k(j+1)}_{p_{2}=k(j)} \alpha^{p_{1},l}_{k}\alpha^{p_{2},1}_{k} h_{l} \left( \intg \sum^{n-1}_{i,j = 1} g^{ij}_{0} \frac{\pa \Phi_{p_{1}}}{\pa \xi_{i}} \frac{\pa \Phi_{p_{2}}}{\pa \xi_{j}} \G \right) \\
&= \frac{n-3}{r} \lamk \sum^{k(j+1)}_{p=k(j)} \sum^{\infty}_{l=1} \alpha^{p,l}_{k}\alpha^{p,1}_{k} h_{l} = \frac{n-3}{4r}\left(\sig_{+}-\sig_{-}\right) \lamk \sum^{k(j+1)}_{p=k(j)}(\alpha^{p,1}_{k})^{2} + o(1).
\end{align*}
Hence the third term is 
\begin{equation*}
\frac{n-3}{4r}\left(\sig_{+}-\sig_{-}\right) \lamk \e \sum^{k(j+1)}_{p=k(j)}(\alpha^{p,1}_{k})^{2} + o(\e).
\end{equation*}
Therefore Theorem \ref{thm3} follows from Lemma \ref{key lemma1}.

\vspace{1.0\baselineskip}
{\bf Acknowledgments.} 
The author would like to thank Professor Shigeru Sakaguchi (Tohoku University), who is his supervisor, for suggesting this interesting problem and for many stimulating discussions. Also the author would like to thank Lorenzo Cavallina (Tohoku University) for warm encouragement and constructive comments. Moreover the author is grateful to the referees for their constructive suggestions.

\end{document}